\documentclass[smallcondensed]{svjour3}     
\smartqed  
\usepackage[mathscr]{eucal}
\usepackage{hyperref}
\hypersetup{
	colorlinks=true,
	citecolor=blue,
	linkcolor=blue,
	filecolor=magenta,      
	urlcolor=blue,
}

\usepackage{multicol, latexsym}
\usepackage{blindtext}
\usepackage{subcaption}
\usepackage{url}
\usepackage{enumitem}
\usepackage{bm}
\usepackage{float}
\usepackage{array}
\usepackage{cite}
\usepackage{chngcntr}
\usepackage{caption}
\usepackage{graphicx}
\usepackage{epstopdf}

\usepackage{textcomp}
\usepackage{amsmath}
\usepackage{amssymb}
\usepackage{epsfig}
\usepackage{tabularx, caption}
\usepackage[round]{natbib}
\usepackage[misc]{ifsym}

\journalname{}
\newtheorem{assumption}{Assumption}
\begin{document}
\bibliographystyle{plainnat}	
	\title{Reduction principle for functionals of vector random fields
		\thanks{\textbf{Supplementary Materials} The codes used for simulations and examples in this article are available in the folder "Research materials" from  \url{https://sites.google.com/site/olenkoandriy/}.
	}}
	\author{Andriy Olenko \and Dareen Omari}
	
	\institute{\Letter\  A. Olenko \at
		\email{a.olenko@latrobe.edu.au}           
		\and
		D. Omari \at
		\email{omari.d@students.latrobe.edu.au}\\
		\at	
		Department of Mathematics and Statistics, La Trobe University,
		\at	
		Melbourne, VIC, 3086, Australia\\
	}

	\date{Received: 30 March 2018 / Accepted: date}
	
	\maketitle
	
	\begin{abstract}
		We prove a version of the reduction principle for functionals of vector long-range dependent random fields. The components of the fields may have different long-range dependent behaviours. The results are illustrated by an application to the first Minkowski functional of the Fisher--Snedecor random fields. Simulation studies confirm the obtained theoretical results and suggest some new problems.
	\end{abstract}
	\section{Introduction}\label{sec1}
	Over the last four decades, a great deal of effort has been devoted to studying the geometric characteristics of excursion sets of random fields. The obtained theoretical results have been utilised in a variety of applications, including in geoscience, astrophysics,
	medical imaging and other related fields (see~\citealt{adler2009random,azais2009level}). Among numerous stochastic models, Gaussian and related  fields (such as $\chi^{2}$, $F$ and $t$ fields) are the most popular in studying excursion sets. The reason for this popularity is their simplicity and mathematical tractability.\

	In various statistical applications, the methods differ with the type of dependence between observations. The dependence properties of a random process are usually characterized using its covariance function. A stationary random process $\eta(x), \ x\in \mathbb R$, is called weakly (or short-range) dependent if its covariance decreases rapidly. More precisely, in this case the covariance function $B(x)=\textbf{Cov}(\eta(x+y),\eta(y))$ is integrable, i.e. 
	$\int_{\mathbb{R}}|B(x)|dx<\infty$. In contrast, the random process $\eta(x)$  possesses long-range (or strong) dependence  if its covariance function decays slowly. In this case, we assume that the covariance function $B(x)$ is non integrable, i.e.
	$\int_{\mathbb{R}}|B(x)|dx=\infty.$
	An alternative definition of long-range dependence is based on singular properties of the spectral density of a random process,  such as unboundedness at zero (see~\citealt{doukhan2002theory,leonenko2013tauberian}). Long-range dependent processes play a crucial role in various scientific disciplines and applied fields, including geophysics, astronomy, agriculture, engineering and economics (see~\citealt{leonenko1999limit,ivanov1989statistical,doukhan2002theory}). In the theory of stochastic processes, long-range dependent models give new types of limit theorems and parameter estimates in comparison to the weak dependent case (see, for example,~\citealt{ivanov1989statistical,worsley1994local,leonenko2013tauberian, beran2013long}).\
	
	The geometric properties of sets in $ \mathbb R^d$ can be described by $d+1$ different Minkowski functionals. 
	Minkowski functionals are important tools in integral and stochastic geometry. They are also widely used in analyzing data sets in many other fields, such as medicine~(\citealt{adler2007applications}), media, image analysis~(\citealt{zhao2010image}), physics and cosmology~(\citealt{marinucci2004testing}). Gaussian random fields are the most frequently used stochastic models in these problems as the moments of Minkowski functionals can be obtained in explicit forms (see~\citealt{tomita1990formation}).
	
	For a random field $\eta(x)$, $x\in T\subset\mathbb{R}^{d}$, the excursion set $A_{T}(a)$ is defined as 
	a set of points $x\in T$, where $\eta(x)$ exceeds some threshold $a\in\mathbb{R}$ (see~\citealt{adler2009random}), i.e. 
	$A_{T}(a)=\lbrace x\in T :\eta(x)\geq a\rbrace$.
	For a smooth random field $\eta(x)$, the excursion set $A_{T}(a)$ decomposes into a finite union of compact sets as $a\rightarrow\infty$. Numerous results for the geometry of the excursion sets of random fields can be found in \citealt{adler2009random} and \citealt{adler2010excursion}. The first Minkowski functional (or sojourn measure) of a random field represents the volume of the excursion set $A_{T}(a)$. It is natural to consider this functional if $T$ is a bounded observation window and to study its asymptotic behaviour when the window's size grows (see~\citealt{adler2010excursion,bulinski2012central}). The first Minkowski functional and its Hermite expansions were discussed in the one-dimensional case for discrete-time processes in~\citealt{doukhan2002asymptotics}. Some recent developments in multidimensional and continuous counterparts can be found in~\citealt{leonenko2014sojourn}.
		
Limit theorems are the central topic in the theory of probability. Considerable attention has been paid to study the asymptotic behaviour of sums (or integrals) of non-linear functionals of stationary Gaussian random processes and fields.
Limit theorems were derived using specific dependence structures imposed on random fields. The central limit theorem (CLT) holds under the usual normalization $n^{-d/2}$ when the summands (integrands) are weakly dependent random fields or processes.  For non-linear functionals of Gaussian random fields, the CLT was proved by~\citealt{breuer1983central}.  Furthermore,~\citealt{de1995central} obtained a generalisation for stationary Gaussian vector processes. Other CLTs were proved for functionals of Gaussian processes and fields in~\citealt{hariz2002limit} and \citealt{kratz2017central}.
\raggedbottom
 \vspace*{-.40cm}
\begin{figure}[H]
\begin{center}
	\centering
		\includegraphics[width=1\linewidth,trim={0 0 0 1.5cm},clip]{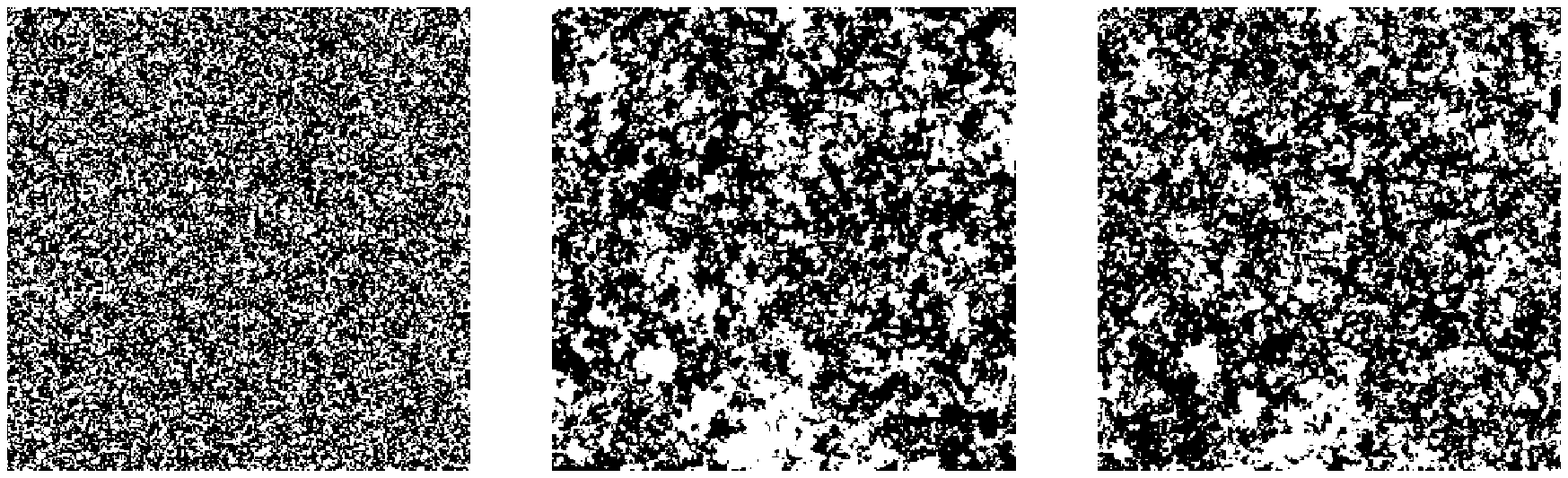}\vspace{-2.2cm} 		\hspace{-1cm}
	\includegraphics[width=1\linewidth]{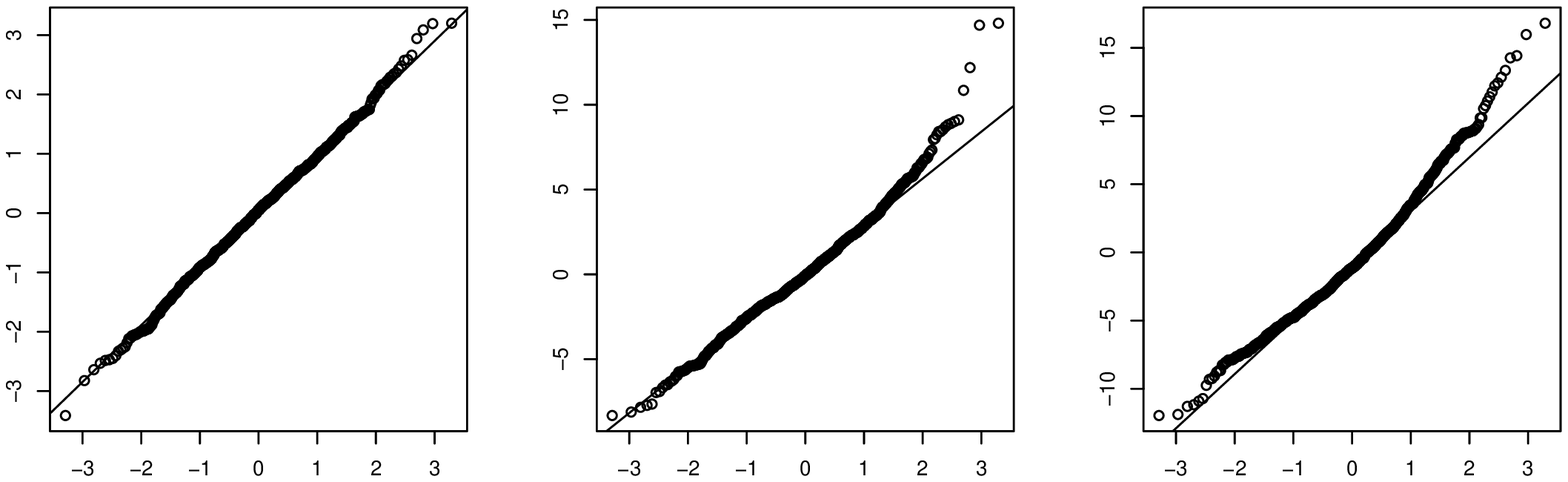} \hspace{1cm} \vspace{-1.2cm} 
	\caption{Two-dimensional excursion sets and corresponding Q–-Q plots.}  
	\label{fig:1a}\vspace{0.6cm}
		\includegraphics[width=0.4\linewidth]{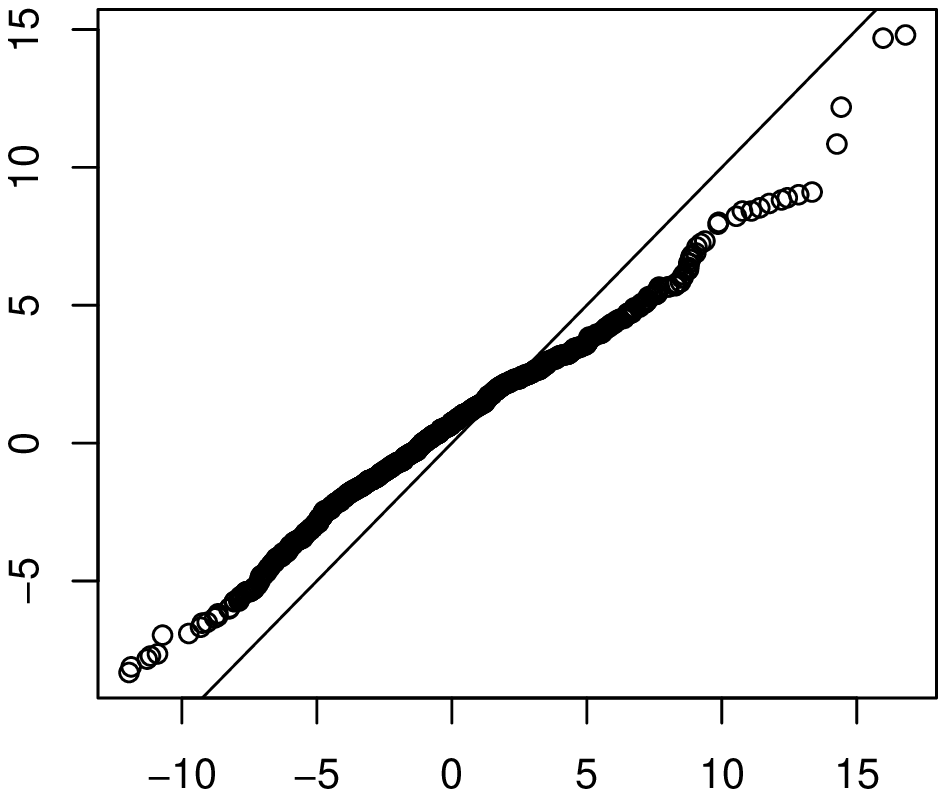} \vspace{-0.6cm} 
	\caption{Q-Q plot for second versus third models.}
	\label{fig:fi2}\vspace{-0.6cm}
\end{center}
\end{figure}
In contrast to the CLT, non-central limit theorems arise in the case of long-range dependence. In this case, one has to use different normalizing coefficients and non-Gaussian limits, which are called Hermite type distributions. It was~\citealt{rosenblatt1961independence} who first observed these situations. The earliest results in this field were obtained by~\citealt{taqqu1975weak}, where the weak limits of partial sums of Gaussian processes were studied using characteristic functions.~\citealt{dobrushin1979non} and~\citealt{taqqu1979convergence} established novel results on the weak convergence in which the limits were obtained in terms of multiple Wiener-It\^{o} integrals. The previous results were generalised for stationary zero-mean Gaussian sequences of vectors by~\citealt{arcones1994limit}. Recently, a new scenario was established by~\citealt{bai2013multivariate}. They studied the multivariate limit theorems for functionals of stationary Gaussian series under long-range dependence, short-range dependence and a mixture of both. Excellent surveys of limit theorems for long-range dependent random fields can be found in~\citealt{anh2015rate,doukhan2002theory,ivanov1989statistical,leonenko1999limit,spodarev2014limit,stoev2007limit}.

Limit theorems for Minkowski functionals of stationary and isotropic Gaussian random fields have been studied under short and long-range dependence assumptions by~\citealt{ivanov1989statistical}. The CLT was proved for broad classes of random fields under different conditions in~\citealt{bulinski2012central,demichev2015functional}. Limit theorems for sojourn measures were discussed for geometric functionals of homogeneous isotropic Gaussian random fields that exhibit long-range dependence in~\citealt{ivanov1989statistical}. Recently~\citealt{leonenko2014sojourn} investigated the limit behaviour for sojourn measures of heavy-tailed random fields (Student and Fisher-–Snedecor) that are weakly or strongly dependent. \

All these publications obtained limit theorems for functionals of vector random fields in which the components possess the same long-range dependent behaviour, see~\citealt{leonenko2014sojourn}.  However, in some applications, it may be preferable to examine cases where the components of a vector random field are different. For example,~\citealt{worsley1994local} assumed that all subjects in positron emission tomography studies can be modelled by the same Gaussian field, but in reality the subject's fields can be different. \

The first row in Figure~{\rm\ref{fig:1a}} shows two-dimensional excursion sets for realizations of three different types of random fields (from left to right):
\begin{itemize}
	\item short-range dependent normal
	scale mixture model with the covariance function  $B(\|x\|)=\mathcal{I}\cdot \mbox{exp}(-\|x\|^{2})$; 
	\item Cauchy model with components having the same long-range dependent behaviour; and
	\item Cauchy model in which the components have different long-range dependent behaviours.
\end{itemize}
More details are provided in Section~{\rm\ref{sec6}}.
The excursion sets are shown in black colour.
The Q–-Q plots in the second row correspond to the models shown above. One thousand simulations and large areas for each model were used to produce these Q-Q plots. Hence, the obtained results are rather close to asymptotic distributions. It is clear that the limit law of the first model is normal, while for the second model the data are not normally distributed. The first and the second models were discussed in detail in~\citealt{leonenko2014sojourn}.
It is also clear from Figure~{\rm\ref{fig:fi2}} and the normal Q–-Q plots in Figure~{\rm\ref{fig:1a}} that the limit law of the third model is not normal and differs from the second model. In this paper, we study the third model and investigate the asymptotic behaviour of Minkowski functionals of vector random fields with components that possess different long-range dependent  behaviours. 

The main result of the paper is the reduction principle for vector random fields with differently distributed long-range dependent components. It is demonstrated how to apply the main result to study the asymptotic behaviour of the Fisher-–Snedecor random fields with heavy-tailed marginal distributions. The obtained results are important not only in deriving asymptotic statistical inference for stochastic processes and random fields, but also because they can be used in applied probability modelling to reduce models' complexity by employing dominant components. 

The paper is organised as follows. In Section~{\rm\ref{sec2}} we outline basic notations and definitions that will be used in the subsequent sections. Section~{\rm\ref{sec3}} presents assumptions and auxiliary results. In Section~{\rm\ref{sec4}} we derive the main  result for functionals of vector long-range dependent random fields. Examples, specifications and discussions of the results of Sections~{\rm\ref{sec3}} and~{\rm\ref{sec4}} are provided in Section~5. Finally, Section~{\rm\ref{sec6}} gives some numerical results using simulations from three different models.	
	\section{Preliminaries}\label{sec2}
	In this section, we present basic notations and definitions of the random field theory, multidimensional Hermite expansions and the first Minkowski functional. Also, we introduce the main stochastic model investigated in the examples, namely, Fisher–-Snedecor random fields. In what follows, $\vert\cdot\vert$ and $\Vert\cdot\Vert$ denote the Lebesgue measure and the Euclidean distance in $\mathbb{R}^{d}$, respectively. It is assumed that all random fields are defined on the same probability space $\left(\Omega,\mathcal{F},\textbf{P}\right)$.
	
	We consider a measurable mean square continuous zero-mean homogeneous isotropic real-valued random field  $\eta (x)$, $x\in\mathbb{R}^{d}$, (see~\citealt{ivanov1989statistical,leonenko1999limit}) with the covariance function
	$$B\left(r\right)= \textbf{Cov}(\eta(x),\eta(y))=\int_{0}^{\infty}Y_{d}\left(rz\right)d\Phi\left(z\right),\quad\quad\quad x,y \in \mathbb{R}^d,$$
	where $r=\rVert x-y \lVert$, the function $Y_{d}\left(\cdot\right)$, $d\geq 1$, is defined by
	$$Y_{d}\left(z\right)=2^{(d-2)/2}\Gamma\left(\dfrac{d}{2}\right)J_{(d-2)/2}(z)z^{(2-d)/2},\quad z\geqslant 0,$$
	and $J_{\nu}(\cdot)$ is the Bessel function of the first kind of order $\nu > -1/2$. The finite measure $\Phi\left(\cdot\right)$ is called the isotropic spectral measure of the random field $\eta\left(x\right)$, $x\in\mathbb{R}^{d}$. 
	
	The random field $\eta(x)$ possesses an absolutely continuous spectrum if there exists a function $\varphi(z)$, $z\in[0,\infty)$, such that
	$$\Phi(z)=2\pi^{d/2}\Gamma^{-1}(d/2)\int_{0}^{z}u^{d-1}\varphi(u)du,\quad\quad\quad u^{d-1}\varphi(u)\in L_{1}([0,\infty)).$$
	The function $\varphi(\cdot)$ is called the isotropic spectral density of the random field $\eta\left(x\right)$.
			
	The cumulative distribution function $H(\cdot)$ and the probability density function $h(\cdot)$ of the field $\eta(x)$, are defined as follows:
	$$H(u):=\textbf{P}(\eta(x)\leq u),\quad\quad H(u)=\int_{-\infty}^{u}h(z)dz,\quad\quad u\in\mathbb{R}.$$
	A random field $\eta\left(x\right)$ with an absolutely continuous spectrum has the following isonormal spectral representation
	$$
	\eta\left(x\right)=\int_{\mathbb{R}^{d}}e^{i\langle\lambda,x \rangle}\sqrt{\varphi(\rVert\lambda\lVert)}W(d\lambda),
	$$
	where $W(\cdot)$ is the complex Gaussian white noise random measure on $\mathbb{R}^{d}$.
Let $\Delta\subset\mathbb{R}^{d}$ be a Jordan-measurable convex bounded set  with $|\Delta| > 0$, and $\Delta$ contains the origin in its interior. Also, assume that $\Delta(r)$, $ r > 0$, is the homothetic image of the
	set $\Delta$, with the centre of homothety in the origin and the coefficient $r > 0$, that is,
	$|\Delta(r)| = r^d |\Delta|$.
	
\begin{definition}
	The first Minkowski functional is defined as
	$$
		M_{r}\left\lbrace \eta\right\rbrace :=|\left\lbrace x\in\Delta(r): \eta(x)>a\right\rbrace |=\int_{\Delta(r)}\chi(\eta(x)>a)dx,
		$$
		where $\chi(\cdot)$ is an indicator function and $a$ is a constant.
\end{definition}
	The functional $M_{r}\left\lbrace \eta\right\rbrace$ has a geometrical meaning, namely, the sojourn measure of the random field $\eta(x)$.\
	The mean and the variance functions of $M_{r}\left\lbrace \eta\right\rbrace$ are defined as follows:
	$$\textbf{E}M_{r}\left\lbrace \eta\right\rbrace=|\Delta|r^{d}\textbf{P}(\eta(x)>a)=|\Delta|r^{d}(1-H(a))$$
	and
	$$\textbf{Var} M_{r}\left\lbrace \eta\right\rbrace=\int_{\Delta(r)}\int_{\Delta(r)}\textbf{P}\left\lbrace \eta(x)>a,\eta(y)>a\right\rbrace dxdy-\left[\textbf{E}M_{r}\left\lbrace \eta\right\rbrace\right]^{2},$$
	or
	$$\textbf{Var}M_{r}\left\lbrace \eta\right\rbrace=\int_{\Delta(r)}\int_{\Delta(r)}Cov\left(\zeta(x),\zeta(y)\right) dxdy,$$
	where $\zeta(x)=\chi(\eta(x)>a)$, $x\in\mathbb {R}^{d}$.\\ 
	Therefore, investigations of the integrals
	$\int_{\Delta(r)}\int_{\Delta(r)}Q\left(\Vert x-y \Vert\right)dxdy$
	are important to analyze the $\textbf{Var} M_{r}\left\lbrace \eta\right\rbrace$. Define two independent random
	vectors $U$ and $V$ that are uniformly distributed inside the set $\Delta(r)$. Then, we have the following representation:
	\begin{align}\label{eq1}
	\int_{\Delta(r)}\int_{\Delta(r)}Q(\Vert x-y \Vert) dxdy &=|\Delta|^{2}r^{2d}\textbf{E}Q(\Vert U-V\Vert)\notag\\
	&=|\Delta|^{2}r^{2d}\int_{0}^ {diam\lbrace\Delta(r)\rbrace}Q(\rho)\psi_{\Delta(r)}(\rho)d\rho,
	\end{align}
	where $\psi_{\Delta(r)}(\rho)$, $\rho\geq 0$, denotes the density function of the distance $\Vert U-V\Vert$ between $U$ and $V$.	
	\begin{lemma}{\rm{(\citealt{peccati2011wiener})}}
		Let $\left(\eta_{1},\ldots,\eta_{2p}\right)$ be a $2p$-dimensional zero-mean Gaussian vector with
		$$ \mathbf{E}
		 \left(\eta_{j}\eta_{k}\right)=
		\begin{cases} 
		1,\quad\textsl{if}\quad k=j,\\
		r_{j},\quad\textsl{if}\quad k=j+p\quad and\quad 1 \leq j\leq p,\\
		0, \quad otherwise, 
		\end{cases}
		$$
		then
		$$
		\\\mathbf{E}\prod_{j=1}^{p} H_{k_{j}}(\eta_{j})H_{m_{j}}(\eta_{j+p})=\prod_{j=1}^{p}\delta_{k_{j}}^{m_{j}}k_{j}!r_{j}^{k_{j}},
		$$
		where $\delta_{k_{j}}^{m_{j}}$ is the Kronecker delta function.
	\end{lemma} 
	Consider
	$$
	e_{v}(\omega)= \prod_{j=1}^{p}H_{k_{j}}(\omega_{j}),
	$$
	where $\omega=(\omega_{1},\dots,  \omega_{p})' \in \mathbb{R}^{p}$, $v=(k_{1},\dots,k_{p})\in \mathbb{Z}^{p},$ and all $k_{j}\geq 0$ for $j = 1,\dots,p.$\
	
	The polynomials $\{e_{v}(\omega)\}_{v}$ form a complete orthogonal system in the Hilbert space 
\[
L_{2}\left(\mathbb{R}^{p},\phi(\lVert\omega\rVert)d\omega\right)=\left \{G: \int_{\mathbb{R}^{p}}G^{2}(\omega)\phi(\lVert\omega\rVert)d\omega<\infty\right\},
\]
	where
	$$\phi(\lVert\omega\rVert)= \prod_{j=1}^{p}\phi(\omega_{j}),\quad\quad \phi(\omega_{j})= \dfrac{e^ {-\omega_{j}^2/2}}{\sqrt{2\pi}}. $$

	An arbitrary function $G(\omega)\in L_{2}\left(\mathbb{R}^{p},\phi(\lVert\omega\rVert)d\omega\right)$ admits an expansion with Hermite coefficients $C_{v}$, given as the following:
	$$
	G(\omega)=\sum_{k=0}^{\infty}\sum_{v\in N_{\kappa}}\dfrac{C_{v}e_{v}(\omega)}{v!},\quad C_{v}=\int_{\mathbb{R}^{p}}G(\omega)e_{v}(\omega)\phi(\lVert\omega\rVert)d\omega,
	$$
	where $v!= k_{1}!\dots k_{p}!$ and $$N_{k}=\left\lbrace(k_{1},\dots,k_{p})\in\mathbb{Z}^{p}: \sum_{j=1}^{p}k_{j}=\kappa,k_{j}\geqslant 0,j=1,\dots,p\right\rbrace.$$\
	
By Parseval's identity, 
	$$\sum_{k=0}^{\infty}\sum_{v\in N_{\kappa}}\dfrac{C_{v}^{2}}{v!}=\int_{\mathbb{R}^{p}}G^{2}(\omega)\phi(\lVert\omega\rVert)d\omega.$$
	\begin{definition} The smallest integer $\kappa\geqslant 1$ such that $C_{v}=0$ for all $v \in N_{j}$ and $j=1,\dots,\kappa-1$, but $C_{v}\neq 0$ for some $v \in N_{\kappa}$ is called the Hermite rank of $G(\cdot)$ and is denoted by $H rank G$.
	\end{definition}

	In this paper, we consider an example of heavy tailed random fields constructed from Gaussian fields. To define these fields, we use a vector random field $\bm{\eta}(x)=[\eta_{1}(x),\dots,\eta_{m}(x)]^{'}$, $ x\in\mathbb{R}^{d},$ with $\textbf{E}\bm{\eta}(x)=0$ and $\eta_{i}(x)$, $i=1,\dots, m $, that are independent homogeneous isotropic unit variance Gaussian random fields. 
	
	\begin{definition} The Fisher--Snedecor random field is defined by
		$$F_{n,m-n}(x)=\dfrac{(\eta_{1}^{2}(x)+\dots+\eta_{n}^{2}(x))/n}{(\eta_{n+1}^{2}(x)+\dots+\eta_{m}^{2}(x))/(m-n)},\quad x\in\mathbb{R}^{d}.
		$$
	\end{definition}
	The random field $F_{n,m-n}(x)$ has a marginal Fisher--Snedecor distribution with the probability density function 
	$$
	h(u)=\dfrac{n^{n/2}(m-n)^{\frac{m-n}{2}}\Gamma(m/2)}{\Gamma(n/2)\Gamma((m-n)/2)}\dfrac{u^{n/2-1}}{(m-n+nu)^{m/2}},\quad u\in [0,\infty),
	$$
	and the cumulative distribution function 
	$$
	H(u)=I_{\frac{nu}{m-n+nu}}\left(\dfrac{n}{2},\dfrac{m-n}{2}\right),
	$$
	where
	$$
	I_{\mu}(p,q)=\frac{\Gamma(p+q)}{\Gamma(p)\Gamma(q)}\int_{0}^{\mu}t^{p-1}(1-t)^{q-1}dt,\quad \mu \in (0,1], p>0, q>0,
	$$
	is the incomplete beta function.
	
	\section{Assumptions and auxiliary results}\label{sec3}
	In this section we introduce some assumptions and auxiliary results.
	\begin{assumption}\label{ass4}
		Let $\bm{\eta}(x)=[\eta_{1}(x),\dots,\eta_{m}(x)]'$, $ x\in\mathbb{R}^{d}$, be a vector homogeneous isotropic Gaussian random field with $\textbf{E}\bm{\eta}(x)=0$ and a covariance matrix $\tilde{B}(x)$ such that 
		\begin{displaymath}
		\tilde{B}(0)=\mathcal{I},\quad
		\tilde{B}_{ij}(\|x\|)=\left\{
		\begin{array}{lr}
		0,\quad\textsl{if}\quad i\neq j,\\
		\|x\|^{-\alpha_{j}}L_{j}\left(\|x\|\right),\quad\textsl{if}\quad i=j ,
		\end{array}
		\right.
		\end{displaymath} 	
		where $i,j\in \{1,\dots,m \}$,  $\alpha_{j}>0$, and $L_{j}\left(\|x\|\right)$ are slowly varying functions at infinity.
	\end{assumption}
	Note that for $\alpha_{j}\in (0,d)$, $j=1,\dots,m,$ the diagonal elements of the covariance matrix $\tilde{B}(x)$ satisfying Assumption~{\rm\ref{ass4}} are not integrable, which corresponds to the case of long-range dependence. \
	
	For simplicity, this paper investigates only the case of uncorrelated components. Generalizations to cross-correlated random field can also be obtained using the approach in~\citealt{leonenko2014sojourn}.  
	
	\begin{assumption}\label{ass2}
		A random field $\eta_{j}\left(x\right)$, $j=1,\dots,m$, has a spectral density $f_{j}\left(\|\lambda\|\right)$, $\lambda\in \mathbb{R}^{d}$, such that
		$$
		f_{j}\left(\|\lambda\|\right)\sim c_{2}\left(d,\alpha_{j}\right)\|\lambda\|^{\alpha_{j}-d}L_{j}\left(\dfrac{1}{\|\lambda\|}\right), \quad\quad\quad \Vert \lambda\Vert\rightarrow 0,
		$$
		where $0<\alpha_{j}<d$ and
		$$c_{2}\left(d,\alpha_{j}\right)=\dfrac{\Gamma\left((d-\alpha_{j})/{2}\right)}{2^{\alpha_{j}}\pi^{d/2}\Gamma\left(\alpha_{j}/2\right)}.$$\
	\end{assumption}
	Denote the Fourier transform of the indicator function of the set $\Delta$ by 
	$$
	\mathcal{K} \left(x\right):=\int_{\Delta}e^{i\langle u,x \rangle}du,\quad x\in\mathbb{R}^{d}.\\
	$$
	\begin{theorem}{\rm{(\citealt{leonenko2014sojourn})}}\label{the2}
		Let $\eta_{j}\left(x\right)$, $ x\in\mathbb{R}^{d}$, $j=1,\dots,m,$ be a homogeneous isotropic Gaussian random field with $\textbf{E}\eta_{j}\left(x\right)=0$. If Assumptions~{\rm\ref{ass4}} and~{\rm\ref{ass2}} hold, then for $r \rightarrow \infty$ the random variables
		$$X_{r,\kappa,j}=r^{\kappa\alpha_{j}/2-d}L_{j}^{-\kappa/2}(r)\int_{\Delta(r)}H_{\kappa}\left(\eta_j(x)\right)dx$$\
		converge weakly to 
		\begin{align}\label{eq3}
			X_{\kappa,j}(\Delta)=c_{2}^{\kappa/2}(d,\alpha_{j})\int_{\mathbb{R}^{d\kappa}}^{\prime}\mathcal{K}\left(\lambda_{1}+\dots+\lambda_{\kappa}\right)\dfrac{W(d\lambda_{1})\dots W(d\lambda_{\kappa})}{\|\lambda_{1}\|^{(d-\alpha_{j})/2}\dots \|\lambda_{\kappa}\|^{(d-\alpha_{j})/2}},
		\end{align} 
		where $\int_{\mathbb{R}^{d\kappa}}^{\prime}$ denotes the multiple Wiener-It\^{o} integral and $W(\cdot)$ is the complex Gaussian white noise random measure on $\mathbb{R}^{d}$.\ 
	\end{theorem}
	\begin{remark}
		The random variable $X_{\kappa,j}(\Delta)$ is given in terms of the multiple Wiener-It\^{o} stochastic integral. For $\kappa = 2$, the probability distribution of $X_{2,j}(\Delta)$ is known as a Rosenblatt-type distribution which is a generalisation of the Rosenblatt distribution to an arbitrary set $\Delta(r)$ (\rm see~\citealt{anh2015rate,taqqu1975weak}).
	\end{remark}

	\begin{lemma}\label{rem3}
		Let $k_{1,l},\dots, k_{m,l}\in \left\lbrace 0,\dots, l\right\rbrace$, $l\in \mathbb{N}$, satisfy the restriction $\sum_{j=1}^{m}k_{j,l}=l.$ If for some $l_{0}\in \mathbb{N}$ it holds $\sum_{j=1}^{m}\alpha_{j}k_{j,l_{0}} < (l_{0}+1)\min_{1\leq j \leq m}(\alpha_{j}) $ then for any $l>l_{0}$ and $\left\lbrace k_{j,l}\right\rbrace$  
		\begin{align}\label{n1}
			\sum_{j=1}^{m}\alpha_{j}k_{j,l}>\sum_{j=1}^{m}\alpha_{j}k_{j,l_{0}}
		\end{align}
		and, moreover, 
		\begin{align*}
			\sum_{j=1}^{m}\alpha_{j}k_{j,l}-\sum_{j=1}^{m}\alpha_{j}k_{j,l_{0}}>\delta,
		\end{align*}
		where the constant 	$\delta:=\delta\left(\{\alpha_{j}\},\{k_{j,l_{0}}\}\right)>0$	
		depends only on $\left\lbrace \alpha_{j} \right\rbrace $ and $\left\lbrace k_{j,l_{0}}\right\rbrace. $
	\end{lemma}
\begin{proof}
	Note that as all $k_{j,l}\geq 0$ then 
	$$\sum_{j=1}^{m}\alpha_{j}k_{j,l}\geq \min_{1\leq j\leq m}(\alpha_{j})\sum_{j=1}^{m}k_{j,l}=l\cdot\min_{1\leq j\leq m}(\alpha_{j}).$$
	By the assumptions on $\{\alpha_{j}\}$ and the inequality $l>l_{0}$
	$$l\cdot\min_{1\leq j\leq m}(\alpha_{j})>\sum_{j=1}^{m}\alpha_{j}k_{j,l_{0}},$$
	which gives~({\rm\ref{n1}}). The constant $\delta$ can be chosen as 
	$$\delta=(l_{0}+1)\min_{1\leq j\leq m}(\alpha_{j})-\sum_{j=1}^{m}\alpha_{j}k_{j,l_{0}},$$
	which completes the proof.\qed
	\end{proof}
	\begin{remark}
		As $\sum_{j=1}^{m}\alpha_{j}k_{j,l_{0}}\leq 
		l_{0}\max_{1 \leq j \leq m}(\alpha_{j})$ then it follows that the statement of Lemma~{\rm\ref{rem3}} is true if $$\dfrac{\max_{1 \leq j \leq m}(\alpha_{j})}{\min_{1 \leq j \leq m}(\alpha_{j})}\leq\dfrac{l_{0}+1}{l_{0}}=1+\dfrac{1}{l_{0}}.$$ 	
	\end{remark}
	In Assumption~{\rm\ref{ass4}} the slowly varying functions $L_{j}(\cdot)$ and the powers $\alpha_{j}$, $j=1,\dots,m,$ can be different, compare to the case of equal $\alpha_{j}=\alpha$ and $L_{j}(\cdot)=L(\cdot)$ in~\citealt{leonenko2014sojourn}. Therefore, one needs the next result about magnitudes of the products $\prod_{j=1}^{m}B_{jj}^{k_{j,l}}(\cdot)$. 
	
	\begin{lemma}\label{lem4}
		If $l_{0}$, $\{\alpha_{j}\}$ and $\left\lbrace k_{j,l_{0}}\right\rbrace$ are selected as in Lemma~{\rm\ref{rem3}} then there is a constant $C$ such that for all $l>l_{0}$ and $\left\lbrace k_{j,l}\right\rbrace $ 
		\begin{align}\label{n2}
			\prod_{j=1}^{m}B_{jj}^{k_{j,l}}(z)\leq C\prod_{j=1}^{m}B_{jj}^{k_{j,l_{0}}}(z), \quad z\geq0.
		\end{align}
	\end{lemma} 
\begin{proof}
	As $B_{jj}(z)\in (0,1]$, it is enough to prove~({\rm\ref{n2}}) for $l=l_{0}+1$.\\
	Using Assumption~{\rm\ref{ass4}} we write
	\begin{align}\label{6}
		\dfrac{\prod_{j=1}^{m}B_{jj}^{k_{j,l_{0}+1}}(z)}{\prod_{j=1}^{m}B_{jj}^{k_{j,l_{0}}}(z)}
		&=\dfrac{\prod_{j=1}^{m}L_{j}^{k_{j,l_{0}+1}}(z)/z^{\alpha_{j}{k_{j,l_{0}+1}}}}{\prod_{j=1}^{m}L_{j}^{k_{j,l_{0}}}(z)/z^{\alpha_{j}{k_{j,l_{0}}}}}\notag\\
		&=\dfrac{\prod_{j=1}^{m}L_{j}^{k_{j,l_{0}+1}}(z)}{\prod_{j=1}^{m}L_{j}^{k_{j,l_{0}}}(z)}z^{\sum_{j=1}^{m}\alpha_{j}{k_{j,l_{0}}}-\sum_{j=1}^{m}\alpha_{j}{k_{j,l_{0}+1}}}.
	\end{align}
	By Lemma~{\rm\ref{rem3}} for each $\left\lbrace k_{j,l_{0}+1}\right\rbrace$ we have
	\begin{align}\label{666}
		\delta\left( \{\alpha_{j}\},\left\lbrace k_{j,l_{0}}\right\rbrace,\left\lbrace k_{j,l_{0}+1}\right\rbrace\right):=\sum_{j=1}^{m}\alpha_{j}k_{j,l_{0}+1}-\sum_{j=1}^{m}\alpha_{j}k_{j,l_{0}}>\delta>0. 
	\end{align}
	Also, by properties of slowly varying functions $$L_{\left\lbrace k_{j,l_{0}}\right\rbrace,\left\lbrace k_{j,l_{0}+1}\right\rbrace}(z)=\frac{\prod_{j=1}^{m}L_{j}^{k_{j,l_{0}+1}}(z)}{\prod_{j=1}^{m}L_{j}^{k_{j,l_{0}}}(z)}$$ is a slowly varying function.
	Hence the expression in~({\rm\ref{6}}) equals $$z^{-\delta\left(\{\alpha_{j}\}, \left\lbrace k_{j,l_{0}}\right\rbrace,\left\lbrace k_{j,l_{0}+1}\right\rbrace\right)}L_{\left\lbrace k_{j,l_{0}}\right\rbrace,\left\lbrace k_{j,l_{0}+1}\right\rbrace}(z).$$
	Note, that $z^{-\delta\left(\{\alpha_{j}\}, \left\lbrace k_{j,l_{0}}\right\rbrace,\left\lbrace k_{j,l_{0}+1}\right\rbrace\right)}L_{\left\lbrace k_{j,l_{0}}\right\rbrace,\left\lbrace k_{j,l_{0}+1}\right\rbrace}(z)$ is bounded on each interval $[a,b]$, $a>0$, $b<+\infty$.\\
	Moreover, by properties of slowly varying functions and~({\rm\ref{666}}) we get 
	\begin{align*}
		\lim_{z\rightarrow \infty}z^{-\delta\left(\{\alpha_{j}\}, \left\lbrace k_{j,l_{0}}\right\rbrace,\left\lbrace
			k_{j,l_{0}+1}\right\rbrace\right)}L_{\left\lbrace k_{j,l_{0}}\right\rbrace,\left\lbrace k_{j,l_{0}+1}\right\rbrace}(z)=0
	\end{align*}
	and
	\begin{align*}
		\lim_{z\rightarrow 0}z^{-\delta\left(\{\alpha_{j}\}, \left\lbrace k_{j,l_{0}}\right\rbrace,\left\lbrace
			k_{j,l_{0}+1}\right\rbrace\right)}L_{\left\lbrace k_{j,l_{0}}\right\rbrace,\left\lbrace k_{j,l_{0}+1}\right\rbrace}(z)=\dfrac{\prod_{j=1}^{m}B_{jj}^{k_{j,l_{0}+1}}(0)}{\prod_{j=1}^{m}B_{jj}^{k_{j,l_{0}}}(0)}=1.
	\end{align*}
	Thus, for each $\left\lbrace k_{j,l_{0}+1}\right\rbrace $, the term $z^{-\delta\left(\{\alpha_{j}\}, \left\lbrace k_{j,l_{0}}\right\rbrace,\left\lbrace
		k_{j,l_{0}+1}\right\rbrace\right)}L_{\left\lbrace k_{j,l_{0}}\right\rbrace,\left\lbrace k_{j,l_{0}+1}\right\rbrace}(z)$ is bounded on $[0,\infty)$ and there is a constant $C_{\left\lbrace k_{j,l_{0}+1} \right\rbrace }$ such that
	$$z^{-\delta\left(\{\alpha_{j}\}, \left\lbrace k_{j,l_{0}}\right\rbrace,\left\lbrace
		k_{j,l_{0}+1}\right\rbrace\right)}L_{\left\lbrace k_{j,l_{0}}\right\rbrace,\left\lbrace k_{j,l_{0}+1}\right\rbrace}(z)\leq  C_{\left\lbrace k_{j,l_{0}+1} \right\rbrace },\quad z\in [0,\infty).$$ 
	For each $l_{0}\in \mathbb{N}$, the number of sets $\left\lbrace k_{j,l_{0}+1} \right\rbrace$ is finite as $\sum_{j=1}^{m}k_{j,l_{0}+1}=l_{0}+1$ and $k_{j,l_{0}+1} \in \mathbb{N}\bigcup\{0\}$. Therefore, the constant $C$ in~({\rm\ref{n2}}) can be selected as $C= \max C_{\left\lbrace k_{j,l_{0}+1} \right\rbrace }<\infty$.\qed
\end{proof}
\section{Main result}\label{sec4}
In this section, we prove a version of the reduction principle for functionals of vector long-range dependent random fields. The result generalises Theorem 4 by~\citealt{leonenko2014sojourn} that studied the case of vector fields having the same type of long-range dependent components. It shows that even for the case of different components, the leading terms at $H rank G$ level determine asymptotic distributions.\

Consider the following two random variables:
$$K_{r}=\int_{\bigtriangleup(r)}G\left(\bm{\eta}\left(x\right)\right)dx\quad and\quad K_{r,\kappa}=\sum_{v\in N_{\kappa}}\dfrac{C_{v}}{v!}\int_{\bigtriangleup(r)}e_{v}\left(\bm{\eta}\left(x\right)\right)dx,$$
where $C_{\kappa}$ are the Hermite coefficients of the function $G(\cdot)$, $v=(k_{1,\kappa},\dots,k_{m,\kappa})$ and $\sum_{i=1}^{m}k_{i,\kappa}=\kappa$.
\begin{theorem}\label{the3}
	Suppose that $\bm{\eta}\left(x\right)$, $ x\in\mathbb{R}^{d}$, satisfies Assumption~{\rm\ref{ass4}}, $H rank G=\kappa\geq 1$, $\sum_{j=1}^{m}\alpha_{j}k_{j,\kappa}\leq (\kappa+1)\min_{1\leq j\leq m}(\alpha_{j})$, and 
	$\sum_{j=1}^{m}\alpha_{j}k_{j,\kappa}<d$.
	If a limit distribution exists for at least one of the random variables
	$$\dfrac{K_{r}}{\sqrt{\textbf{Var}K_{r}}}\quad and\quad \dfrac{K_{r,\kappa}}{\sqrt{\textbf{Var}K_{r,\kappa}}},$$
	then the limit distribution of the other random variable exists as well, and the limit distributions coincide when $r\rightarrow \infty$. 	
\end{theorem}
\begin{proof}
Let	
$$V_{r}=\sum_{l \geq \kappa+1}\sum_{v\in N_{l}}\dfrac{C_{v}}{v!}\int_{\bigtriangleup(r)}e_{v}\left(\bm{\eta}\left(x\right)\right)dx,$$
then
$$\textbf{Var} \left( K_{r}\right) =\textbf{Var} \left( K_{r,\kappa}\right)  + \textbf{Var} \left( V_{r}\right) .$$
The term $\textbf{Var} \left( K_{r,\kappa}\right)$ can be estimated as 
\begin{align*}
\textbf{Var}\left(  K_{r,\kappa}\right) &=\textbf{Var} \sum_{v\in N_{\kappa}}\dfrac{C_{v}}{v!}\int_{\Delta(r)}e_{v}\left(\bm{\eta}\left(x\right)\right)dx\\
&=\textbf{Var} \sum_{v\in N_{\kappa}}\dfrac{C_{v}}{v!}\int_{\Delta(r)}\prod_{j=1}^{m}H_{k_{j,\kappa}}(\eta_{j}(x))dx \\
&= \sum_{v\in N_{\kappa}}\dfrac{C_{v}^{2}}{v!}\int_{\Delta(r)}\int_{\Delta(r)} \textbf{E} \prod_{j=1}^{m}H_{k_{j,\kappa}}(\eta_{j}(x))H_{k_{j,\kappa}}(\eta_{j}(y))dxdy
\end{align*}
\begin{align*}
&= \sum_{v\in N_{\kappa}}\dfrac{C_{v}^{2}}{v!}\int_{\Delta(r)}\int_{\bigtriangleup(r)}\prod_{j=1}^{m}\left( \Vert x-y \Vert^{-\alpha_{j}} L_{j}(\Vert x-y \Vert)\right) ^{k_{j,\kappa}}dxdy\\
&= \vert \Delta(r) \vert^{2} \sum_{v\in N_{\kappa}}\dfrac{C_{v}^{2}}{v!}\int_{0}^{r\cdot diam\left\lbrace \Delta\right\rbrace} \prod_{j=1}^{m}z^{-\alpha_{j} k_{j,\kappa}} L_{j}^{k_{j,\kappa}}(z) \psi_{\Delta(r)}(z)dz\\
&= \vert \Delta(r) \vert^{2} \sum_{v\in N_{\kappa}}\dfrac{C_{v}^{2}}{v!}
\int_{0}^{r\cdot diam \left\lbrace \Delta\right\rbrace}  z ^{- \sum_{j=1}^{m}\alpha_{j}k_{j,\kappa}}
\prod_{j=1}^{m}L_{j}^{k_{j,\kappa}}(z)\psi_{\Delta(r)}(z)dz\\
&= \vert \Delta \vert^{2}  \sum_{v\in N_{\kappa}}\dfrac{C_{v}^{2}}{v!}r^{2d- \sum_{j=1}^{m}\alpha_{j}k_{j,\kappa}}\int_{0}^{diam \left\lbrace \Delta\right\rbrace}  z ^{- \sum_{j=1}^{m}\alpha_{j}k_{j,\kappa}}\\
&\times\prod_{j=1}^{m} L_{j}^{ k_{j,\kappa}}(zr)    \psi_{\Delta}(z)dz.
\end{align*} 
As a product of slowly varying functions is a slowly varying function too, then, by Theorem 2.7 in~\citealt{seneta1976functions}, we obtain as $r \rightarrow \infty$
\begin{equation}\label{77}
\textbf{Var} \left( K_{r,\kappa}\right) = \vert \Delta \vert^{2} \sum_{v\in N_{\kappa}}\dfrac{C_{v}^{2}}{v!}  c_{1}\left( \left\lbrace \alpha_{j}\right\rbrace ,\left\lbrace k_{j,\kappa}\right\rbrace ,\Delta\right) r^{2d- \sum_{j=1}^{m}\alpha_{j}k_{j,\kappa}}\prod_{j=1}^{m}L_{j}^{k_{j,\kappa}}(r)  (1+\small o(1)), 
\end{equation}
where $c_{1}\left( \left\lbrace \alpha_{j}\right\rbrace ,\left\lbrace k_{j,\kappa}\right\rbrace ,\Delta\right)=\int_{0}^{diam \left\lbrace \Delta\right\rbrace} z ^{- \sum_{j=1}^{m}\alpha_{j}k_{j,\kappa}} \psi_{\Delta}(z)dz.$

Note that all coefficients $c_{1}\left( \left\lbrace \alpha_{j}\right\rbrace ,\left\lbrace k_{j,\kappa}\right\rbrace ,\Delta\right)$ are correctly defined as  
\begin{align*}
c_{1}\left( \left\lbrace \alpha_{j}\right\rbrace ,\left\lbrace k_{j,\kappa}\right\rbrace ,\Delta\right)&=\int_{0}^{diam\{\Delta\}} z^{- \sum_{j=1}^{m}\alpha_{j}k_{j,\kappa}}
\psi_{\Delta}(z)dz\\
&=\vert \Delta \vert^{-2} \int_{\Delta}\int_{\Delta} \Vert x-y \Vert^{-\sum_{j=1}^{m}\alpha_{j}k_{j,\kappa}}dxdy\\
&\leq \vert \Delta \vert^{-1} \int_{0}^{diam\left\lbrace \Delta\right\rbrace } \rho^{d-\left(1+ \sum_{j=1}^{m}\alpha_{j}k_{j,\kappa}\right)} d\rho<\infty.
\end{align*}
Now, for $\textbf{Var} \left( V_{r}\right) $ we obtain 
\begin{align}\label{7}
\textbf{Var} \left( V_{r}\right) &=\vert \Delta(r) \vert^{2}\sum_{l \geq \kappa+1}\sum_{v=(k_{1,l},\dots,k_{m,l})\in N_{l}}\dfrac{C_{v}^{2}}{v!}\int_{0}^{r\cdot 
	diam \left\lbrace \Delta\right\rbrace} \prod_{j=1}^{m} B_{jj}^{k_{j,l}}(z)  \psi_{\Delta(r)}(z)dz\notag\\
&=\vert \Delta\vert^{2} \sum_{l \geq \kappa+1}\sum_{v\in N_{l}}\frac{C_{v}^{2}}{v!}r^{2d}\int_{0}^{diam \left\lbrace \Delta\right\rbrace} \prod_{j=1}^{m} B_{jj}^{k_{j,l}}(rz)  \psi_{\Delta}(z)dz.
\end{align}
For each $l\geq\kappa+1$ and $v=(k_{1,l},\dots,k_{m,l})\in N_{l}$, it is always possible to find $(k_{1,\kappa+1},\dots,k_{m,\kappa+1})\in N_{\kappa}$ such that $k_{i,\kappa+1}\leq k_{i,l}$, $i=1,\dots, m$. Let us denote such $\{k_{j,\kappa+1}\}$ by $\{\tilde{k}_{j,l}\}$. Then we get $\sum_{j=1}^{m}\tilde{k}_{j,l}=\kappa+1$. As $B_{jj}(\cdot)\leq1$, we can estimate the expression in~({\rm\ref{7}}) as follows
\begin{align*}
\textbf{Var} \left( V_{r}\right) \leq\vert \Delta\vert^{2}r^{2d} \sum_{v=(k_{1,l},\dots,k_{m,l})\in N_{l}}\frac{C_{v}^{2}}{v!}\int_{0}^{diam \left\lbrace \Delta\right\rbrace} \prod_{j=1}^{m} B_{jj}^{\tilde{k}_{j,l}}(rz)  \psi_{\Delta}(z)dz. 
\end{align*}  To investigate the above integral we split it into two integrals:

\begin{align*}
&\int_{0}^{diam \left\lbrace \Delta\right\rbrace} \prod_{j=1}^{m} B_{jj}^{\tilde{k}_{j,l}}(rz)  \psi_{\Delta}(z)dz
=\int_{0}^{r^{-\beta}} \prod_{j=1}^{m} B_{jj}^{\tilde{k}_{j,l}}(rz)  \psi_{\Delta}(z)dz\\
&+\int_{r^{-\beta}}^{diam \left\lbrace \Delta\right\rbrace} \prod_{j=1}^{m} B_{jj}^{\tilde{k}_{j,l}}(rz)  \psi_{\Delta}(z)dz
=:I_{1}+I_{2},
\end{align*} 
where $\beta\in (0,1)$.\\
By Lemma~{\rm\ref{lem4}} the first integral $I_{1}$ can be estimated as follows 
\begin{align*}
I_{1}&=\prod_{j=1}^{m} B_{jj}^{{k}_{j,\kappa}}(r) \int_{0}^{r^{-\beta}}\prod_{j=1}^{m}\dfrac{ B_{jj}^{\tilde{k}_{j,l}}(rz)}{ B_{jj}^{{k}_{j,\kappa}}(r)}   \psi_{\Delta}(z)dz\\
&\leq C\prod_{j=1}^{m} B_{jj}^{{k}_{j,\kappa}}(r) \int_{0}^{r^{-\beta}}\prod_{j=1}^{m} \dfrac{B_{jj}^{{k}_{j,\kappa}}(rz)}{B_{jj}^{{k}_{j,\kappa}}(r)}  \psi_{\Delta}(z)dz\\
&= C\prod_{j=1}^{m} B_{jj}^{{k}_{j,\kappa}}(r) \int_{0}^{r^{-\beta}}z^{-\sum_{j=1}^{m}\alpha_{j}k_{j,\kappa}}\prod_{j=1}^{m} \dfrac{L_{j}^{{k}_{j,\kappa}}(rz)}{L_{j}^{{k}_{j,\kappa}}(r)}  \psi_{\Delta}(z)dz.
\end{align*}
Using the estimate (20) in~\citealt{leonenko2014sojourn} we obtain 
\begin{align*}
I_{1}\leq C\prod_{j=1}^{m} B_{jj}^{{k}_{j,\kappa}}(r)\dfrac{\sup_{u\in (0,r)}u^{\epsilon}\prod_{j=1}^{m}L_{j}^{{k}_{j,\kappa}}(u)}{r^{\epsilon}\prod_{j=1}^{m}L_{j}^{{k}_{j,\kappa}}(r)}\int_{0}^{r^{-\beta}}z^{-\epsilon}z^{-\sum_{j=1}^{m}\alpha_{j}k_{j,\kappa}}\psi_{\Delta}(z)dz,
\end{align*}
where $\epsilon$ is an arbitrary positive constant.\

As $\prod_{j=1}^{m}L_{j}^{{k}_{j,\kappa}}(u)$ is a slowly varying function, by Theorem 1.5.3~\citealt{bingham1989regular} we obtain that
$$\lim_{r\rightarrow\infty}   \dfrac{\sup_{u\in (0,r)}u^{\epsilon}\prod_{j=1}^{m}L_{j}^{{k}_{j,\kappa}}(u)}{r^{\epsilon}\prod_{j=1}^{m}L_{j}^{{k}_{j,\kappa}}(r)} =1.$$ 
By (21) in~\citealt{leonenko2014sojourn} we get the following estimate
\begin{align*}
\int_{0}^{r^{-\beta}}z^{-\epsilon}z^{-\sum_{j=1}^{m}\alpha_{j}k_{j,\kappa}}\psi_{\Delta}(z)dz\leq
C r^{-\beta(d-(\epsilon+\sum_{j=1}^{m}\alpha_{j}k_{j,\kappa}))}.
\end{align*}
Now, for the second integral $I_{2}$ we obtain
\begin{align*}
I_{2}&=\int_{r^{-\beta}}^{diam\{\Delta\}}\prod_{j=1}^{m}\dfrac{L_{j}^{\tilde{k}_{j,l}}(rz)}{(rz)^{\sum_{j=1}^{m}\alpha_{j}\tilde{k}_{j,l}}}\psi_{\Delta}(z)dz=\int_{r^{-\beta}}^{diam\{\Delta\}}\prod_{j=1}^{m}\dfrac{L_{j}^{\tilde{k}_{j,l}}(rz)}{L_{j}^{k_{j,\kappa}}(rz)}\\
&\times\dfrac{\prod_{j=1}^{m}L_{j}^{k_{j,\kappa}}(rz)}{(rz)^{\sum_{j=1}^{m}\alpha_{j}k_{j,\kappa}}} (rz)^{{\sum_{j=1}^{m}\alpha_{j}(k_{j,\kappa}-\tilde{k}_{j,l}})} \psi_{\Delta}(z)dz.
\end{align*} Note, that by properties of slowly varying functions
$\tilde{L}(z):=\prod_{j=1}^{m}\frac{L_{j}^{\tilde{k}_{j,l}}(z)}{L_{j}^{k_{j,\kappa}}(z)}$ and $\tilde{L}_{0}(z):=\prod_{j=1}^{m}L_{j}^{k_{j,\kappa}}(z)$
are slowly varying functions. Also, by Lemma~{\rm\ref{rem3}} 
\[\sum_{j=1}^{m} \alpha_j(\tilde{k}_{j,l}-k_{j.k}) \ge \delta(\{\alpha_j\},\{k_{j,\kappa}\}):=\delta_1>0.\]
Therefore, for any $\delta_{2}>0$ we get
\begin{align*}
I_{2}&=\int_{r^{-\beta}}^{diam\{\Delta\}} \tilde{L}(rz) (rz)^{-\delta_{1}}\dfrac{\tilde{L}_{0}(rz)}{(rz)^{\sum_{j=1}^{m}\alpha_{j}k_{j,\kappa}}}
\psi_{\Delta}(z)dz\\
&\leq \dfrac{\tilde{L}_{0}(r)}{r^{\sum_{j=1}^{m}\alpha_{j}k_{j,\kappa}}}\dfrac{\sup_{s\in(r^{1-\beta},r\cdot diam\{\Delta\})}s^{\delta_{2}}\tilde{L}_{0}(s)}{r^{\delta_{2}}\tilde{L}_{0}(r)}\\ &\times\sup_{s\in(r^{1-\beta},r\cdot diam\{\Delta\})}\dfrac{\tilde{L}(s)}{s^{\delta_{1}}}\int_{0}^{diam\{\Delta\}}z^{-\sum_{j=1}^{m}\alpha_{j}k_{j,\kappa}-\delta_{2}}\psi_{\Delta}(z)dz.
\end{align*}
By Theorem 1.5.3 in~\citealt{bingham1989regular},  $\lim_{r\rightarrow\infty}   \dfrac{\sup_{s\in (r^{1-\beta},r\cdot diam\{\Delta\})}s^{\delta_{2}}\tilde{L}_{0}(s)}{r^{\delta_{2}}\tilde{L}_{0}(r)}$ is finite.\\
It follows from (22) in~\citealt{leonenko2014sojourn} that for any $\delta_{2}\in (0,\delta_{1})$
\begin{align*}
\sup_{s\in\left(r^{1-\beta},r.diam\{\Delta\}\right)}\dfrac{\tilde{L}(s)}{s^{\delta_{1}}}=o\left(r^{(\delta_{2}-\delta_{1})(1-\beta)}\right).
\end{align*}
Note, that $\int_{0}^{diam\{\Delta\}}z^{-\sum_{j=1}^{m}\alpha_{j}k_{j,\kappa}-\delta_{2}}\psi_{\Delta}(z)dz<\infty$ if $\delta_{2}\in \left(0,d-\sum_{j=1}^{m}\alpha_{j}k_{j,\kappa}\right)$.
Thus, 
\begin{align*}
I_{2}\leq C\dfrac{\tilde{L}_{0}(r)}{r^{\sum_{j=1}^{m}\alpha_{j}k_{j,\kappa}}}o(r^{(\delta_{2}-\delta_{1})(1-\beta)})=C\prod_{j=1}^{m}B_{jj}^{k_{j,\kappa}}(r)\cdot o\left(r^{(\delta_{2}-\delta_{1})(1-\beta)}\right).
\end{align*}
Finally, combining the above results we obtain
\begin{align*}
\textbf{Var}\left( V_{r}\right) &\leq\vert\Delta\vert^{2}r^{2d}\sum_{v=(k_{1,l},\dots,k_{m,l})\in\mathbb{N}_{l}}\dfrac{C_{v}^{2}}{v!}\left(I_{1}+I_{2}\right)
\leq C\vert\Delta\vert^{2}r^{2d}\sum_{v\in\mathbb{N}_{l}}\dfrac{C_{v}^{2}}{v!}\\
&\times\prod_{j=1}^{m}B_{jj}^{k_{j,\kappa}}(r)\left(r^{-\beta(d-(\epsilon+\sum_{j=1}^{m}\alpha_{j}k_{j,\kappa}))}+o\left( r^{-(\delta_{1}-\delta_{2})(1-\beta)}\right)\right) .
\end{align*}
For small enough $\epsilon$, $\delta_{1}$, $\delta_{2}$ $(\delta_{1}>\delta_{2}>0)$, we can make the powers of the two summands in the parentheses negative.

Noting that $\prod_{j=1}^{m}B_{jj}^{k_{j,\kappa}}(r)=\prod_{j=1}^{m}L_{j}^{k_{j}}(r)/r^{\sum_{j=1}^{m}\alpha_{j}k_{j,\kappa}}$ and comparing with~({\rm\ref{77}}) we obtain $\lim_{r\rightarrow\infty}{\textbf{Var}(V_{r})}/{\textbf{Var}(K_{r,k})}=0$.

It means that $K_{r,k}$ completely determines the asymptotic distribution of $K_{r}$, which completes the proof.\qed	
\end{proof} 
\section{Examples and applications}\label{sec5}
In this section we discuss differences of the obtained results and ones in~\citealt{leonenko2014sojourn}. The framework of dominant components is introduced. We also show how the obtained results can be applied to investigate asymptotic behaviour of Minkowski functionals of Fisher-–Snedecor random fields.\

The statement of Theorem~{\rm\ref{the3}} looks similar to Theorem 4 in~\citealt{leonenko2014sojourn}. However, the case of different $B_{jj}(\cdot)$ is more complex. Indeed,~\citealt{leonenko2014sojourn} considered $B_{jj}(z)=\dfrac{L(z)}{z^{\alpha}}$, for all $j=1,\dots,m$, therefore, for any m-tuple $(k_{1,\kappa},\dots,k_{m,\kappa})$ it held that $\prod_{j=1}^{m}B_{jj}^{k_{j,\kappa}}(z)=\prod_{j=1}^{m}\dfrac{L^{k_{j,\kappa}}(z)}{z^{\alpha k_{j,\kappa}}}=\dfrac{L^{\kappa}(z)}{z^{\alpha \kappa}}$ as $\sum_{j=1}^{m}k_{j,\kappa}=\kappa$. Hence, the same normalizing factor $r^{2d-\alpha \kappa}L^{\kappa}(r)$ was used for each term $\int_{\Delta(r)}e_{v}\left(\bm{\eta}\left(x\right)\right)dx$, $v\in N_{\kappa}$, in $K_{r,\kappa}$. However, if, as in this paper, $B_{jj}(\cdot)$ are different then normalizing factors of $\int_{\Delta(r)}e_{v}\left(\bm{\eta}\left(x\right)\right)dx$ can vary depending on $v \in \mathbb{N}_{\kappa}$. Thus, it is not enough to use the same normalization for all terms at $H rank G$ level, but all asymptotics for individual terms at $H rank G$ level must also be studied.

 We will use the following result from~\citealt{leonenko2014sojourn}.
\begin{theorem}
Let  $\bm{\eta} (x)=[\eta _{1}(x),\ldots ,\eta _{m}(x)]^{\prime }$, $x\in \mathbb{R}^{d},$ satisfy Assumption~{\rm\ref{ass4}} with $\alpha_{j}=\alpha$, $\alpha \in (0,d/2)$, and $L_{j}(r)=L(r)$, $j=1,\dots,m$. Then, for $r\to \infty,$
	the distribution of the random variable
	\[\frac{M_{r}\left\{ F_{n,m-n}\right\}-\left| \Delta \right| r^{d}\left(1-I_{\frac{na}{m-n+na}}\left(\frac{n}{2},\frac{m-n}{2}\right)\right)}{c_4(a,n,m)\,r^{d-\alpha}L(r)}\]
	converges to the distribution of the random variable
	\[\frac{X_{2,1}+...+X_{2,n}}{n}-\frac{X_{2,n+1}+...+X_{2,m}}{m-n},\]
	where $X_{2,j}$, $j=1,\dots,m,$ are independent copies of  $X_{2}(\Delta)$ defined by~\rm({\rm\ref{eq3}}) and
	\begin{align*}
	c_{4}(a,n,m)=\frac{(na/m-n)^{n/2}\Gamma(m/2)}{(1+na/m-n)^{m/2}\Gamma((m-n)/2)\Gamma(n/2)}.
	\end{align*}	
\end{theorem}
	First, we demonstrate that for the case of components with different $B_{jj}(\cdot)$, $j=1,\dots , m$, there might be no convergence at all.

	\begin{theorem}\label{th44}
There is a vector random field $\bm{\eta}(x)$ satisfying Assumption~{\rm\ref{ass4}}, such that for any normalization of $M_{r}\left\{ F_{n,m-n}\right\}-\left| \Delta \right| r^{d}\left(1-I_{\frac{na}{m-n+na}}\left(\frac{n}{2},\frac{m-n}{2}\right)\right)$ 
the limit does not exist or is a degenerated random variable talking values $0$ or $\infty$, when $r\rightarrow \infty$.
\end{theorem}
\begin{proof}
		Note that in this case $G(\omega)=\chi\left(\frac{\frac{1}{n}(\omega_{1}^2+\dots+\omega_{n}^2)}{\frac{1}{m-n}(\omega_{n+1}^2+\dots+\omega_{m}^2)} >a\right) $ and $C_{v}=0$ when $v\in \mathbb{N}_{1}$. Also, it was shown in~\citealt{leonenko2014sojourn} that $H rank G=2$ for $v\in N_{2}$ and
	\begin{displaymath}
	C_{v}=\left\{
	\begin{array}{lr}
	0,\quad \quad \quad\quad\quad \mbox{if} \quad \mbox{all} \quad k_{j}<2 \quad \mbox{for} \quad j\in \{1,\dots,m\},\\
	\frac{2c_{4}(a,n,m)}{n},\quad\quad\mbox{if}\quad  k_{j}=2\quad \mbox{for}\quad \mbox{some} \quad j \in \{ 1, \dots , n\} ,\\
	-\frac{2c_{4}(a,n,m)}{m-n},\quad\mbox{if}\quad k_{j}=2\quad \mbox{for} \quad \mbox{some} \quad j \in  \{n+1, \dots , m\} .
	\end{array}
	\right.
	\end{displaymath} 
Therefore, by Theorem~{\rm\ref{the3}} to study the asymptotic of $$M_{r}\left\{ F_{n,m-n}\right\}-\left| \Delta \right| r^{d}\left(1-I_{\frac{na}{m-n+na}}\left(\frac{n}{2},\frac{m-n}{2}\right)\right)$$ one has to study the behaviour of 
\begin{align}\label{eq9}
K_{r,2}(n,m)=2c_{4}(a,n,m)\bigg[\frac{1}{n}\sum_{j=1}^{n}\varsigma_{j}(r)-\frac{1}{m-n}\sum_{j=n+1}^{m}\varsigma_{j}(r)\bigg]
\end{align}
where $\varsigma_{j}(r)=\int_{\Delta(r)}H_{2}(\eta_{j}(x))dx=\int_{\Delta(r)}(\eta_{j}^{2}(x)-1)dx$.

Suppose that in Assumption~{\rm\ref{ass4}} all $\alpha_{j}$,  $j=1,\dots,m$, are equal, i.e. there exist $\alpha \in \left(0,\frac{d}{\kappa}  \right) $ such that $\alpha_{j}=\alpha$. Using Theorem~{\rm\ref{the2}} for $\kappa=2$ we obtain
\begin{align*}
\frac{1}{r^{d-\alpha}L_{j}(r)}\int_{\Delta(r)}H_{2}(\eta(x))dx\rightarrow  X_{2}\left(\Delta\right),\quad \quad r\rightarrow \infty,
\end{align*}
where
$X_{2}\left(\Delta\right)=c_{2}(d,\alpha)\int_{{\mathbb{R}}^{2d}}^{\prime}\mathcal{K}(\lambda_{1}+\lambda_{2})\frac{W(d\lambda_{1})W(d\lambda_{2})}{\Vert\lambda_{1}\Vert^{(d-\alpha)/2}\Vert\lambda_{2}\Vert^{(d-\alpha)/2}}.$

Let $L_{1}(r)=\exp{\left((\log{r})^{1/3}\cos\left(\log{r} \right)^{1/3}\right)}$ and $L_{j}(r)=1$, $j=2,\dots,m$. Note, that by~\citealt{bingham1989regular}, p.16, $L_{1}(\cdot)$ is a slowly varying function that satisfies
\begin{align}\label{eq10}
\liminf\limits_{r\rightarrow \infty} L_{1}(r)=0 \quad \mbox{and}\quad \limsup\limits_{r\rightarrow \infty} L_{1}(r)=+\infty.
\end{align}
Now,
\begin{align*}
\dfrac{K_{r,2}(n,m)}{r^{d-\alpha}}
&=2c_{4}(a,n,m)\bigg[\dfrac{L_{1}(r)}{n}\dfrac{\varsigma_{1}(r)}{r^{d-\alpha}L_{1}(r)}
+\frac{1}{r^{d-\alpha}}\bigg(\frac{1}{n}\sum_{j=2}^{n}\varsigma_{j}(r)\\
&-\frac{1}{m-n}\sum_{j=n+1}^{m}\varsigma_{j}(r)\bigg)\bigg]\\
&=:2c_{4}(a,n,m)\left(\dfrac{L_{1}(r)}{n}U_{1}(r)+U_{2}(r)\right).
\end{align*}
If $r \rightarrow \infty$ then $U_{1}(r)$ converges to $X_{2,1}(\Delta)$ and $U_{2}(r)$ converges to $\frac{1}{n}\sum_{j=2}^{n} X_{2,j}(\Delta)-\frac{1}{m-n}\sum_{j=n+1}^{m} X_{2,j}(\Delta)$, where $X_{2,j}(\Delta)$, $j=1,\dots,m$, are independent copies of $X_{2}(\Delta)$. However, by~({\rm\ref{eq10}}) the product $\dfrac{L_{1}(r)}{n}U_{1}(r)$ is not convergent. Hence, there is no a normalization for $K_{r,2}(n,m)$ that makes it convergent to a non-degenerated random variable.\qed
\end{proof}
Theorem~{\rm\ref{th44}} demonstrates that for differently distributed $\eta_{j}$ in Assumption~{\rm\ref{ass4}} the functional $K_{r}$ can be divergent when $r \rightarrow \infty$ regardless of what normalization was chosen.
\begin{definition}\label{d4}
	Let us call a component $\eta_{j}(x)$ of $\bm{\eta}(x)$ a dominant component if for all $i\neq j$ either  $\alpha_{i}=\alpha_{j}$ and there exists $\lim_{r\rightarrow\infty}\frac{L_{i}(r)}{L_{j}(r)}=a_{ij}<\infty$ or $\alpha_{i}<\alpha_{j}$.
\end{definition}
\begin{remark}
	By Definition~{\rm\ref{d4}} there can be several dominant components. In particular, in the case of $\alpha_{i}=\alpha$ and $L_{i}(\cdot)=L(\cdot)$, for all $i=1,\dots,m$, all components are dominant.
\end{remark}
Now we give an equivalent of Theorem 7 in~\citealt{leonenko2014sojourn} for the general case of components satisfying Assumption~{\rm\ref{ass4}}.
\begin{theorem}\label{thrm4}
	Let $\bm{\eta}(x)$, $x\in\mathbb{R}^{d}$, satisfy Assumptions~{\rm\ref{ass4}} and~{\rm\ref{ass2}} and $\max_{1\leq j \leq m}(\alpha_{j})\leq \frac{3}{2}\min_{1\leq j \leq m}(\alpha_{j})$. Let $\eta_{j_{1}^{*}}(x),\dots, \eta_{j_{N}^{*}}(x)$, $ j_{1}^{*},\dots,j_{N}^{*}\in \{1,\dots,m\}$, $j_{1}^{*}<\dots<j_{N}^{*}$, $N\leq m,$ be all dominant components of $\eta(x)$. Then
	\begin{align}\label{9}
	\frac{M_{r}\left\{ F_{n,m-n}\right\}-\left| \Delta \right| r^{d}\left(1-I_{\frac{na}{m-n+na}}\left(\frac{n}{2},\frac{m-n}{2}\right)\right)}{c_4(a,n,m)\,r^{d-\alpha_{j_{1}^{*}}}L_{j_{1}^{*}}(r)}
	\end{align} 
	converges weakly to the random variable 
	$\sum_{j\in \{j_{1}^{*},\dots,j_{N}^{*}\} } q_{j}X_{2,j}(\Delta)$, where 
	\begin{displaymath}
	q_{j}=\left\{
	\begin{array}{lr}
	\frac{a_{jj_{1}^{*}}}{n},\quad \quad \rm{if} \quad 1\leq j \leq n,\\
	-\frac{a_{jj_{1}^{*}}}{m-n},\quad\rm{if}\quad n+1\leq j \leq m .
	\end{array}
	\right.
	\end{displaymath} 
\end{theorem} 
\begin{proof} First, let us note that if $\bm{\eta}(x)$ has two dominant components $\eta_{j_{1}^{*}}$ and $\eta_{j_{2}^{*}}$, $j_{1}^{*}\neq j_{2}^{*}$, then $\alpha_{j_{1}^{*}}=\alpha_{j_{2}^{*}}$ and $\lim_{r\rightarrow\infty}\frac{L_{j_{1}^{*}}(r)}{L_{j_{2}^{*}}(r)}=a_{j_{1}^{*}j_{2}^{*}}\in (0,\infty)$.

By Theorem~{\rm\ref{the3}}, the random variable~({\rm\ref{9}})  and $\dfrac{K_{r,2}}{c_4(a,n,m)\,r^{d-\alpha_{j_{1}^{*}}}L_{j_{1}^{*}}(r)}$ have the same asymptotic behaviour.\

By~({\rm\ref{eq9}}), it follows that   
\begin{align*}
\dfrac{K_{r,2}(n,m)}{c_4(a,n,m)\,r^{d-\alpha_{j_{1}^{*}}}L_{j_{1}^{*}}(r)}= 2\sum_{j=1}^{m}q_{j} \dfrac{r^{d-\alpha_{j}}}{r^{d-\alpha_{j_{1}^{*}}}}\dfrac{L_{j}(r)}{L_{j_{1}^{*}}(r)}\dfrac{1}{r^{d-\alpha_{j}}L_{j}(r)}\int_{\Delta(r)}\left(\eta_{j}^{2}(x)-1 \right) dx.
\end{align*}
Note that $\frac{r^{d-\alpha_{j}}}{r^{d-\alpha_{j_{1}^{*}}}}\frac{L_{j}(r)}{L_{j_{1}^{*}}(r)}$ converges to $a_{jj_{1}^{*}}$ if the $j^{th}$ component is dominant and converges to 0 otherwise. As $\frac{1}{r^{d-\alpha_{j}}L_{j}(r)}\int_{\Delta(r)}\left(\eta_{j}^{2}(x)-1 \right) dx$ converges to $X_{2,j}(\Delta)$, when $r\rightarrow\infty$, we obtain the statement of the theorem.\qed
\end{proof}	
\section{Numerical results}\label{sec6}
In this section we present numerical studies to investigate the asymptotic behaviour of Minkowski functionals of vector random fields. We consider cases where components of $\bm{\eta}(x)$ possess same or different long-range dependent behaviours. The simulation studies confirm the obtained theoretical results and suggest some new problems.  

The following models of $\bm{\eta} (x)=[\eta _{1}(x),\eta _{2}(x) ,\eta _{3}(x)]^{\prime }$, $x\in \mathbb{R}^{2},$ were used for simulations :
\begin{itemize}
\item[(a)] Cauchy model with components having the same long-range dependent behaviour;
\item[(b)] Cauchy model in which the components have different long-range dependent behaviours; and
\item[(c)] Bessel model with components having the same cyclic long-memory behaviour. 
\end{itemize}
The Cauchy covariance function is 
\[B(\|x\|)=  (1+\|x\|^{2})^{-\frac{\alpha}{2}},\quad 0 < \alpha < 1,\] while the Bessel one is \[B(\|x\|)=2^{\nu}\Gamma(\nu+1)\frac{J_{\nu}(\|x\|)}{\|x\|^{\nu}},\quad 0\leq\nu<\frac{1}{2}.\]
For the given range of parameters, both covariance functions are non-integrable, i.e. $\int_{{\mathbb{R}}^{2}} \rvert B(\|x\|)\lvert dx=\infty$. Hence, the long-range dependent case is investigated.
Note, that if for all components $\eta_{i}(x)$ in model (b) $\alpha_{i}=\alpha$, then model (b) coincides with model (a).

The R software package RandomFields (see~\citealt{Schlather}), was used to simulate $\eta_{i}(x)$, $x\in \mathbb{R}^{2}$, $i=1,2,3,$ from the above models.

To investigate limit behaviours the following procedure was used. For 1000 simulated realizations of Fisher–-Snedecor random fields $F_{1,2}(x)=\dfrac{2 \eta_{1}^{2}(x)}{\eta_{2}^{2}(x)+\eta_{3}^{2}(x)}$, $x \in \mathbb{R}^2$, areas of excursion sets were computed for each realisation.
The excursion sets above the level $a=1$ are shown in black colour in Figures~{\rm\ref{fig:1a}} and~{\rm\ref{fig:5a}}.
To compared empirical distributions of the excursion areas with the normal law the normal Q-Q plots are presented in Figures~{\rm\ref{fig:1a}} and~{\rm\ref{fig:5b}}. 

Empirical distributions of sojourn measures for the above models were investigated in the following cases.

\textbf{Case 1.} The values $\alpha_{1}=0.65$, $\alpha_{2}=0.8$ and $\alpha_{3}=0.9$ were used. These values satisfy the conditions of Theorem~{\rm\ref{thrm4}}. First, realizations of Fisher–-Snedecor random fields $F_{1,2}(x)$, $x \in \mathbb{R}^2$, were simulated using $\eta_{i}(x)$, $i=1,2,3$, with $\alpha_{i}$ given above. Then, another set of realisations of $F_{1,2}(x)$ was simulated using the same $\alpha$ $(\alpha_{1}, \alpha_{2}$ or $\alpha_{3})$ for all $\eta_{i}(x)$. Q-Q plots of the realizations with different $\alpha_{i}$ versus the realisations with the same $\alpha$ were produced for each $\alpha=\alpha_{i}$.
These Q--Q plots in Figure~{\rm\ref{fig:case1}} indicate that the distributions are close only in the case of $\alpha_{1}=0.65$ (Figure~3a) which corresponds to the dominant component $\eta_{1}(x)$. In two other cases the Q--Q plots may suggest that the distributions are similar, but their variances are different. 

 \textbf{Case 2.} The values $\alpha_{1}=0.1$, $\alpha_{2}=0.5$ and $\alpha_{3}=0.9$ were used. These values do not satisfy the conditions of Theorem~{\rm\ref{thrm4}}. Q-Q plots analogues to Case 1 were obtained and shown in Figure~{\rm\ref{fig:case2}}. The asymptotic distributions are different in all cases. Hence, the reduction principle may not work in this case.
  \begin{figure}[H]
 	\centering
 	\includegraphics[width=1\linewidth]{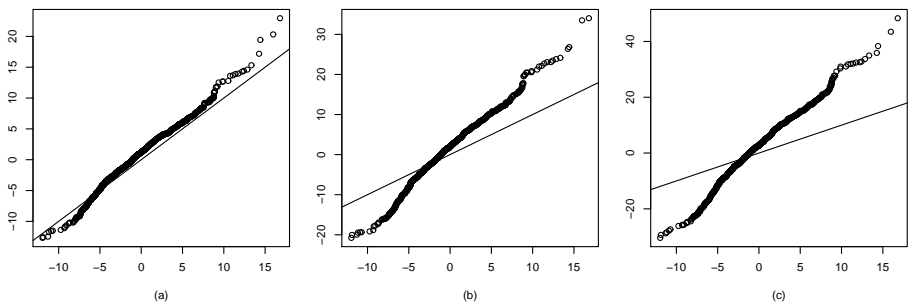} \vspace{-1cm} 
 	\caption{Q--Q plots of model (b) versus three models (a) for case 1. }
 	\label{fig:case1}
 	 \end{figure}
  
 \textbf{Case 3.} Three independent copies of the Cauchy random field with $\alpha=0.5$ were used for models (a) and three independent Bessel random fields with $\nu=0$ for model (c). The corresponding Cauchy and Bessel covariance functions have the same asymptotic hyperbolic decay rate $\|x\|^{-0.5}$, which determines their long-range dependent behaviours. In addition, the Bessel covariance exhibits decaying oscillations as $J_{0}(\|x\|)\sim \sqrt{\frac{2}{\pi \|x\|}}cos\left( \|x\|+\frac{\pi}{4}\right) $, when $\|x\|\rightarrow \infty$. It is due to the cyclic properties of model (c).
In this case the asymptotic distributions of $M_{r}\{F_{1,2}(x)\}$ for models (a) and (c) are different as shown in Figure~{\rm\ref{fig:5c}}. Figure~{\rm\ref{fig:5b}} suggests that the asymptotic normality may be appropriate. That is, new limit theorems are required for  cyclic long-range dependent vector random fields.
    \begin{figure}[H]
    	\vspace{-0.4cm}
	\centering
	\includegraphics[width=1\linewidth]{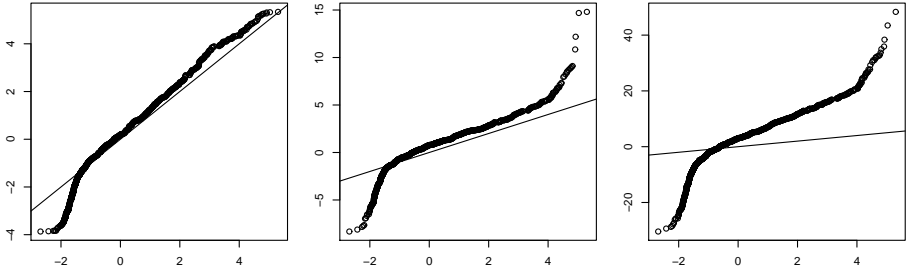} \vspace{-1.2cm}
	\caption{Q–-Q plots of model (b) versus three models (a) for case 2.}\label{fig:case2}
\end{figure}
\vspace{-1.4cm}
 \begin{figure}[H]
	\centering
	\begin{subfigure}[b]{0.35\textwidth}
		\includegraphics[width=.97\textwidth,trim={3cm 0 0 0},clip]{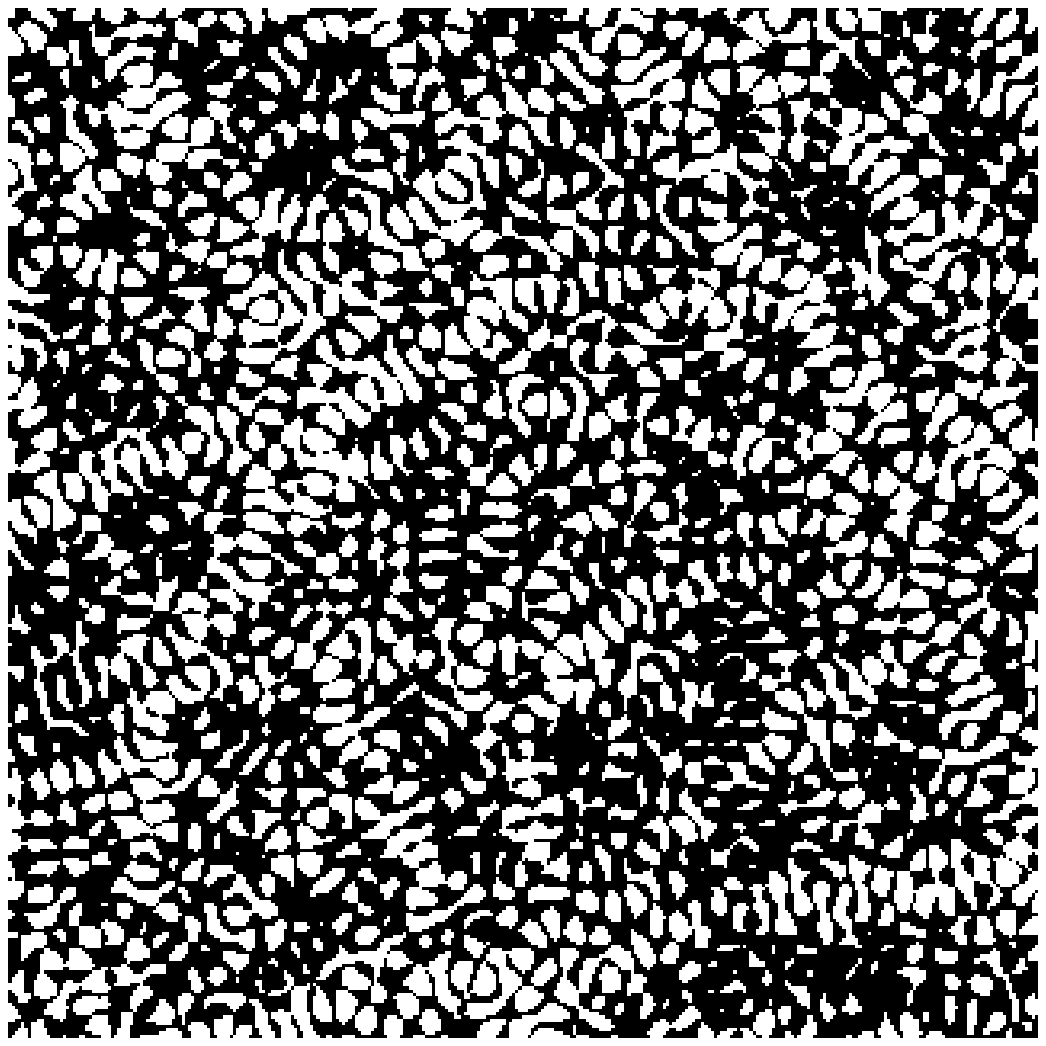}  \vspace{0.05cm}
		\caption{}
		\label{fig:5a}
	\end{subfigure}
	\hfill  \hspace{-3.1cm}
	\begin{subfigure}[b]{0.37\textwidth}
		\includegraphics[width=.97\textwidth, height=4.75cm,trim={1cm 0 0 0},clip]{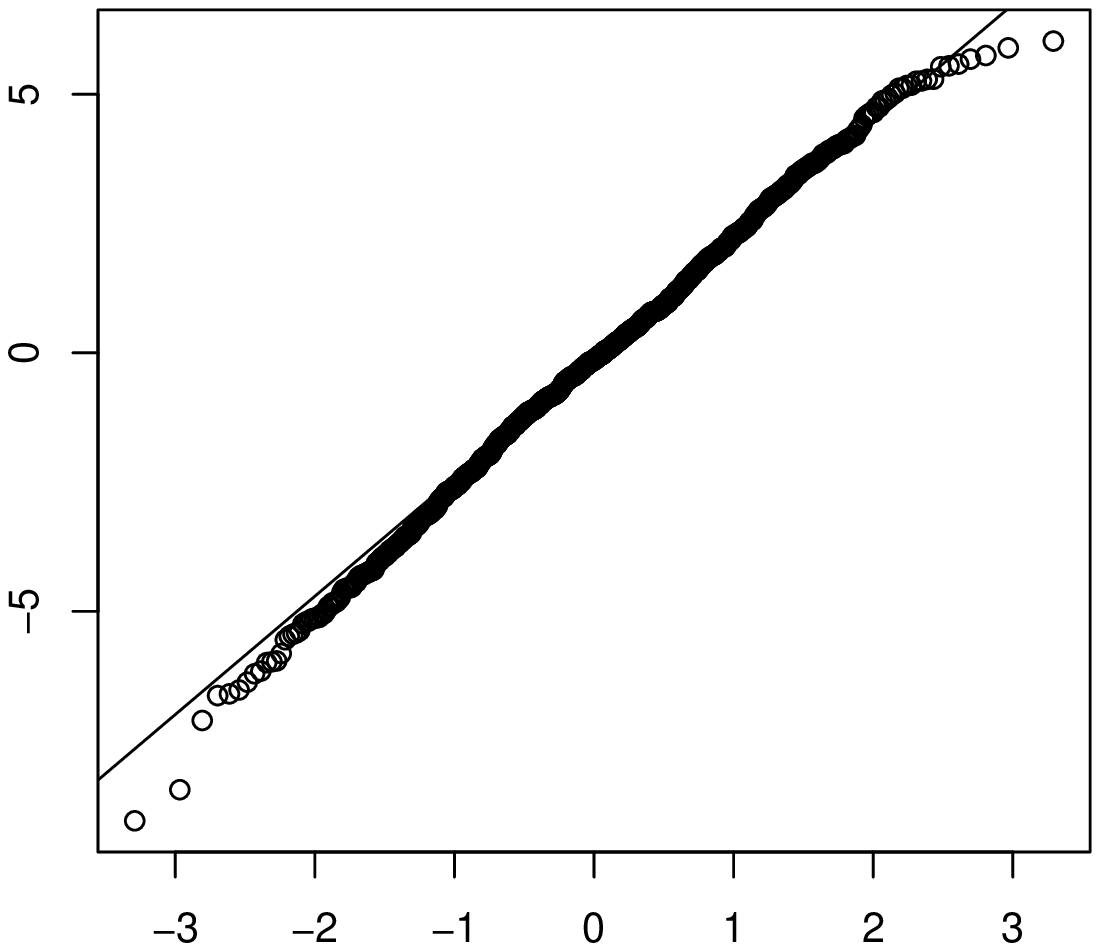} \vspace{-.7cm}
		\caption{}
		\label{fig:5b}
	\end{subfigure}
	 \hfill  \hspace{-2.9cm}
	\begin{subfigure}[b]{0.37\textwidth}
		\includegraphics[height=4.85cm]{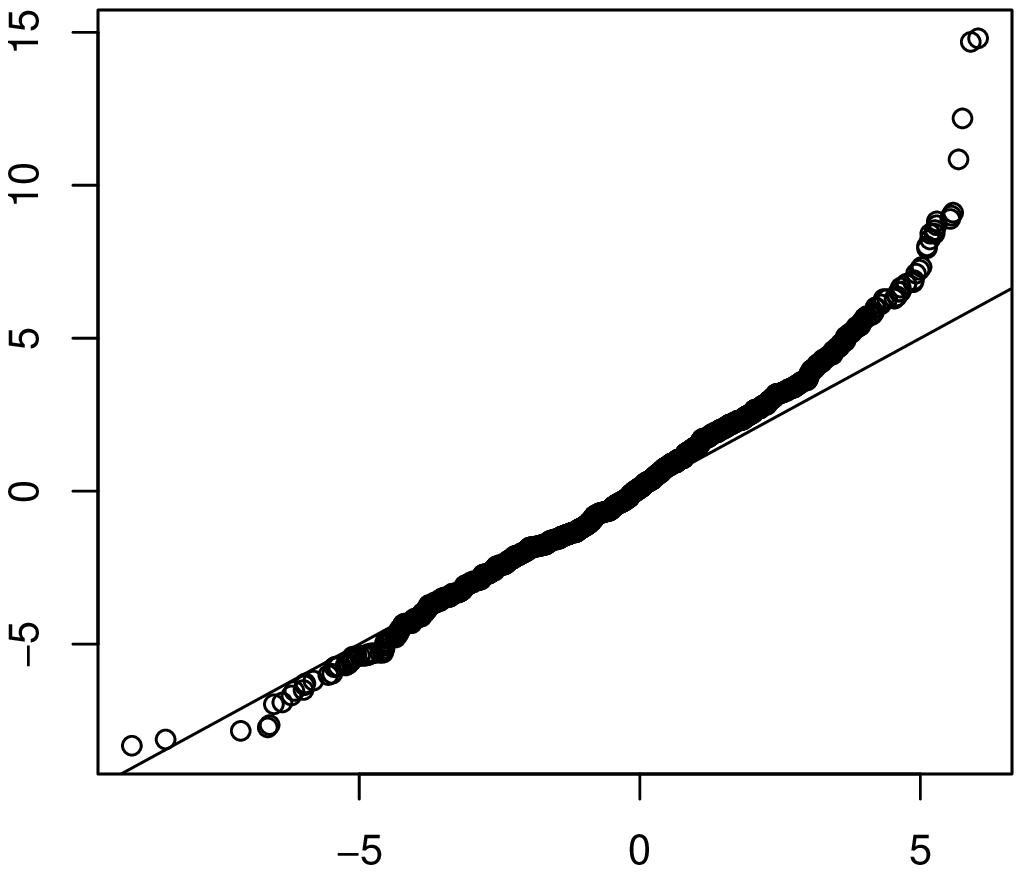} \vspace{-1.1cm}
		\caption{}
		\label{fig:5c}
	\end{subfigure}
	\caption{Excursion set, Q--Q plots for model (c) and for  model (a) versus model (c).}\label{fig:case3}\vspace{-0.2cm}
\end{figure}
 	\section{Conclusions}\label{sec7}
	In this paper, we studied vector random fields with strongly dependent components. The main result is the reduction principle for the case when different components can have different long-range behaviours. 
	It was also investigated how the dominant terms at $H rank G$ level determine asymptotic distributions. We used these results to study the asymptotic behaviour of the first Minkowski functional for the  Fisher-–Snedecor random fields.
	
	In our future research we plan to extend these results to other random fields, in particular, to investigate the vector case with weak and strong dependent components and the case of cyclic/seasonal long-range dependent components.\\
	
\noindent\textbf{Acknowledgements} Andriy Olenko was partially supported under the Australian Research Council's Discovery Projects funding scheme (project DP160101366) and the La Trobe University DRP Grant in Mathematical and Computing Sciences.
	

\end{document}